\documentclass{amsart}

\usepackage{amscd,amssymb,verbatim,graphicx,stmaryrd,amsthm,color,amsmath} 
\usepackage[colorlinks]{hyperref}
\usepackage{quiver}
\usepackage{tikz}
\usepackage[new]{old-arrows}

\input xy    
\xyoption{curve}  
\xyoption{arc} 
\xyoption{frame} 
\xyoption{tips} 
\xyoption{arrow} 
\xyoption{color} 

\usepackage{mathpazo} 

\newtheorem{thm}{Theorem}
\newtheorem{cor}{Corollary}

\newtheorem{theorem}{Theorem}[section]  
\newtheorem{lemma}[theorem]{Lemma}  
\newtheorem{proposition}[theorem]{Proposition}
\newtheorem{corollary}[theorem]{Corollary}
\newtheorem{remark}[theorem]{Remark}
\newtheorem{definition}[theorem]{Definition}
\newtheorem{example}[theorem]{Example}
\newtheorem*{remark*}{Remark}
\newtheorem*{definition*}{Definition}

\newtheorem{innercustomgeneric}{\customgenericname}
\providecommand{\customgenericname}{}
\newcommand{\newcustomtheorem}[2]{%
  \newenvironment{#1}[1]
  {%
   \ifdefined\crefalias\crefalias{innercustomgeneric}{#2}\fi
   \renewcommand\customgenericname{#2}%
   \renewcommand\theinnercustomgeneric{##1}%
   \innercustomgeneric
  }
  {\endinnercustomgeneric}%
  \ifdefined\crefname\crefname{#2}{#2}{#2s}\fi
}

\newcustomtheorem{customprop}{Property}

\newcommand{\bb}[1]{\mathbb{#1}}

\newcommand{\defeq}{\mathrel{\mathpalette{\vcenter{\hbox{$:$}}}=}}

\newcommand{\SU}{\mathrm{SU}}
\newcommand{\SO}{\mathrm{SO}}
\newcommand{\U}{\mathrm{U}}
\newcommand{\GL}{\mathrm{GL}}
\newcommand{\SL}{\mathrm{SL}}

\newcommand{\PSL}{\mathrm{PSL}}

\newcommand{\A}{{\mathcal{A}}}
\newcommand{\V}{{\mathcal{V}}}
\newcommand{\E}{{\mathcal{E}}}
\newcommand{\G}{{\mathcal{G}}}

\newcommand{\Hom}{{\mathrm{Hom}}}

\newcommand{\R}{{\mathcal{R}}}
\newcommand{\Spin}{{\mathrm{Spin}}}

\title{Representation varieties of RAAGs}

\author[A.\ Bao]{A.\ Bao}
\address[Bao]{Department of Mathematics, University of Maryland, College Park, MD 20742}
\email{abao12@terpmail.umd.edu}

\author[A.\ Chakraborty]{A.\ Chakraborty}
\address[Chakraborty]{Mathematical Sciences Department, George Mason University, Fairfax, VA 22030}
\email{achakra5@gmu.edu}

\author[D.~L.~Duncan]{D.~L.~Duncan}
\address[Duncan]{Department of Mathematics \& Statistics, James Madison University, Harrisonburg VA 22801}
\email{duncandl@jmu.edu}

\author[J.\ Larson]{J.\ Larson}
\address[Larson]{Department of Mathematics, Purdue University, West Lafayette, IN 47907}
\email{tlarson.jordan@gmail.com}

\author[K.\ McBride]{K.\ McBride}
\address[McBride]{Department of Mathematics and Statistics, Binghamton University, Binghamton, NY 13902}
\email{kelsondmcbride@gmail.com}

\begin{document}

\begin{abstract}
We investigate the $G$-representation varieties of right-angled Artin groups (RAAGs) for various Lie groups $G$. We show these varieties are connected for a large class of such $G$, including $\SU(n), \mathrm{Sp}(n)$ and $\U(n)$, while they are generally not connected for other large classes, such as $\SO(n)$ and $\Spin(n)$ for $n \geq 3$. When $G = \SO(3)$ we determine the number of connected components of the representation variety associated to any RAAG that is also a 3-manifold group.
\end{abstract}

\maketitle

Right-angled Artin groups (RAAGs) arise in various areas of mathematics from topology to geometric group theory. From a geometric point of view, RAAGs are prototypical examples of CAT(0) groups and serve
as toy examples that mirror many important properties of and inform conjectures about more complicated groups, such as mapping class groups \cite{CW}\cite{Kob}. RAAGs also play a key role in the study of three–manifold topology through, e.g., Agol's virtual fibering criterion \cite{Agol2}, which culminated in his resolution of the virtual Haken conjecture \cite{Agol}. There is a natural way to associate to any (finite) graph $K$ a RAAG $\Gamma_K$, and we review this in Section \ref{sec:DefinitionsAndBasicConstructions}.

Given a finitely-presented group $\pi$ and a topological group $G$, one can form the \emph{$G$-representation variety} $\R(\pi, G)$, consisting of all homomorphisms from $\pi$ to $G$. This too has emerged as a powerful tool in many areas of mathematics. For example, low-dimensional topologists have observed that even relatively simple topological features of $\R(\pi_1(X), G)$, such as its number of connected components, can detect subtle features of a manifold $X$. An early instance of this can be seen in Casson's invariant \cite{AM}, while a somewhat more recent instance is Zentner's result \cite{Zent} that representation varieties can distinguish the standard 3-sphere among all homology 3-spheres. 

Our focus here is on the case where $\pi = \Gamma_K$ is the RAAG associated to a graph $K$. We seek to determine the number of connected components of $\R(K, G) \defeq \R(\Gamma_K, G)$. This number depends on both $K$ and $G$ but, as our analysis shows, this dependence on $G$ is trivial when $G$ satisfies the following property.

\begin{customprop}{A}\label{A} 
\emph{If $S \subseteq G$ is a finite set, then the inclusion $Z(G) \hookrightarrow C_G(S)$ induces a surjection on $\pi_0$.}
\end{customprop}

Here $Z(G)$ is the center of $G$ and $C_G(S)$ is the centralizer of $S\subseteq G$. We will give a complete characterization of the compact, connected, semisimple Lie groups that satisfy Property \ref{A}: Such a Lie group satisfies Property \ref{A} if and only if it is a product of $\SU(n)$'s and $\mathrm{Sp}(n)$'s (see Propositions \ref{prop:SU} and \ref{prop:SemiSimple}). Thus, $\SO(n)$ does \emph{not} satisfy Property \ref{A} for $n \geq 3$. Our first main result motivates the relevance of Property \ref{A} to the connectedness of the $G$-representation variety. 

\begin{thm}\label{thm:1}
If $G$ is connected and satisfies Property \ref{A}, then $\R(K, G)$ is connected for all finite graphs $K$. In particular, $\R(K, \SU(n))$, $\R(K, \mathrm{Sp}(n))$, and $\R(K, \U(n))$ are each connected.
\end{thm}

The proof is given in Section \ref{sec:Proof1}. This extends to arbitrary finite graphs a result of Florentino--Lawton  \cite[Lem. 5.16]{FL2}, who were working in the context of trees.

We therefore turn our attention to groups $G$ that \emph{do not} satisfy Property \ref{A}. We focus primarily on the case where $G = \SO(3)$. As we will see, the number of connected components of $\R(K, \SO(3))$ depends in a subtle way on the combinatorial features of $K$, and we do not know a general formula for this number. However, in some cases a reasonable count can be had. 

\begin{thm}\label{thm:2}
If $K$ is a tree or a cycle, then $\R(K, \SO(3))$ has $2^{\vert \E \vert}$ connected components, where $\E$ is the edge set of $K$. 
\end{thm}

The proof of this theorem is given in Section \ref{sec:Proofs} and is predicated on a comparison between the components of $\R(K, \SO(3))$ and the edge markings of $K$ by elements of $\pi_1(\SO(3)) \cong \left\{ \pm 1 \right\}$. This comparison is brought about by a certain obstruction map $o_K: \R(K, \SO(3)) \rightarrow \left\{ \pm 1 \right\}^{\E}$, which is related to the second Stiefel--Whitney class for principal $\mathrm{SO}(n)$-bundles; see Section \ref{sec:ASW}. Indeed, there is a close link between the representation variety $\R(\pi_1(X), G)$ and the space $\A_{\mathrm{flat}}(P)$ of flat connections on a given principal $G$-bundle $P \rightarrow X$. In terms of principal bundles, we have the following.

\begin{thm}\label{thm:3}
Suppose $X$ is a smooth manifold with $\pi_1(X) \cong \Gamma_K$ a RAAG for a graph $K$ that is either a tree or a cycle. Let $P \rightarrow X$ be a principal $\SO(3)$-bundle (or an oriented real vector bundle of rank 3). Then the quotient $\A_{\mathrm{flat}}(P)/\G(P)$ is connected when it is nonempty, where $\G(P)$ is the group of bundle automorphisms of $P$ covering the identity. 
\end{thm}

We give a general discussion of bundles in Section \ref{sec:Bundles}, where we also indicate precisely when $\A_{\mathrm{flat}}(P)$ is nonempty. In particular, when $X$ is a 3-manifold we can couple our results with a theorem of Droms \cite{Droms} to get:

\begin{cor}\label{cor:4}
Suppose $X$ is a 3-manifold with $\pi_1(X)$ a RAAG. If $P \rightarrow X$ is a principal $\SO(3)$-bundle (or oriented real vector bundle of rank 3), then the space $\A_{\mathrm{flat}}(P) / \G(P)$ is connected and not empty. 
\end{cor}

Proofs of Theorem \ref{thm:3} and Corollary \ref{cor:4} will be given in Section \ref{sec:Proofs}. 

\begin{remark*}
(a) One might ask whether Theorem \ref{thm:2} extends to other groups without Property \ref{A}, such as $\SO(n)$ for $n \geq 4$. Though some of our analysis does extend, much of it does not. More specifically, our proof of Theorem \ref{thm:2} relies heavily on an understanding of the commutator map in the universal cover $\widetilde{G} = S^3$ of $G = \SO(3)$; see Lemmas \ref{lem:Com} and \ref{lem:AntiCom}. Another feature we use is that noncentral elements of $\widetilde{G}$ have abelian centralizers. This latter feature, for example, is not true of $\SO(n)$ for $n \geq 4$.

\medskip

(b) Wang \cite{Wang} has given explicit faithful representations of RAAGs into various groups of higher rank.
\end{remark*}

\medskip

\textbf{Acknowledgements}: This work was completed as part of an REU hosted at James Madison University (JMU) and supported by the National Science Foundation (NSF) through NSF Grant Number DMS 2349593. We are grateful to both JMU and the NSF for their support. We would also like to thank Sean Lawton for his helpful suggestions.

\section{Definitions and basic constructions}\label{sec:DefinitionsAndBasicConstructions}

Suppose $K = (\V, \E)$ is a simple graph with vertex set $\V$ and edge set $\E$; we will always assume $K$ is a finite graph. If $e \in \E$ is an edge, we will write $v(e)$ and $w(e)$ for the vertices to which $e$ is incident; since there are no multi-edges, we can identify $e$ with the pair $\left\{v(e), v(w) \right\}$. Define $\Gamma_K$ to be the group with a generator $x_v$ for each $v \in \V$ and the relation $x_{v} x_{w} = x_{w} x_{v}$ for each edge $\left\{v, w \right\} \in \E$. When $K$ is the empty graph (no vertices and no edges), we take $\Gamma_K$ to be the trivial group. We call $\Gamma_K$ the \emph{right-angled Artin group (RAAG)} associated to $K$. 

For example, if $K$ has no edges, then $\Gamma_K$ is a free group. At the other extreme, when $K$ is a complete graph, $\Gamma_K$ is a free Abelian group. As such, RAAGs interpolate between free groups and free Abelian groups. 

The study of RAAGs goes back at least to Braudisch \cite{Brau}, where they are referred to as \emph{semi-free groups}. Other names have also been used; for example, Droms \cite{Droms} refers to them as \emph{graph groups}.

\medskip

Suppose $\pi$ is a finitely presented group and $G$ is a topological group. The  \emph{$G$-representation variety} of $\pi$ is the set 
$$\R(\pi, G) \defeq \Hom(\pi, G)$$ 
of all homomorphisms from $\pi$ to $G$. We equip this with the compact-open topology relative to the discrete topology on $\pi$. The $G$-representation variety gets its name from the case where $G$ is an algebraic group, for then $\R(\pi, G)$ is naturally a variety. Here we will be primarily interested in the topological features of $\R(\pi, G)$ and so we do not require that $G$ be an algebraic group; see \cite[\S 2.2]{FL} for a discussion of related matters. 

Restrict attention to the case where $\pi = \Gamma_K$ is a RAAG and set $\R(K, G) = \R(\Gamma_K, G)$ as above. The map
\begin{equation}\label{eq:repemb}
\R(K, G) \longrightarrow G^\V, \hspace{1cm} \rho \longmapsto (\rho(v))_{v \in \V}
\end{equation}
is a topological embedding with image equal to the set of tuples $(g_v)_{v \in \V}$ satisfying $[g_{v}, g_{w}]  = 1$ for all $\left\{v, w \right\} \in \E$. Here $[g, h] = gh g^{-1} h^{-1}$ is the group commutator and $1$ is the identity. We will identify $\R(K, G)$ with its image under (\ref{eq:repemb}).

\section{Groups with Property \ref{A}}\label{sec:GroupsWithA}

\subsection{Examples}

Let's start with a straight-forward case.

\begin{proposition}
If $G$ is abelian, then $G$ has Property \ref{A}.
\end{proposition}
\begin{proof}
When $G$ is abelian $Z(G) = G = C_G(S)$ for all $S \subseteq G$. Thus, the inclusion $Z(G) \hookrightarrow C_G(S)$ is the identity map, which is certainly $\pi_0$-surjective. 
\end{proof}

\begin{proposition}
If $G_0$ and $G_1$ each have Property \ref{A}, then so does $G_0 \times G_1$.
\end{proposition}

\begin{proof}
Let $S \subseteq G_0 \times G_1$, and write $S = \left\{(s_{01}, s_{11}), \ldots, (s_{0n}, s_{1n}) \right\}$ for $s_{ij} \in G_i$. Fix $g \in Z_{G_0 \times G_1}(S)$ and write $g = (g_0, g_1)$. Define $S_i \defeq \left\{ s_{ij} \right\}_j$, which is a finite subset of $G_i$. By assumption there is a path $\gamma_i: [0, 1] \rightarrow Z_{G_i}(S_i)$ from $g_i$ to $Z(G_i)$. Then $\gamma(t) \defeq (\gamma_0(t), \gamma_1(t))$ commutes with each element of $S$, and so defined a path in $Z_{G_0 \times G_1}(S)$ from $g$ to $Z(G_0 \times G_1) = Z(G_0) \times Z(G_1)$. 
\end{proof}

Next we will show that each of $\SU(n), \mathrm{Sp}(n)$, and $\U(n)$ have Property \ref{A}. Writing $E_{\lambda}(A)$ for the $\lambda$-eigenspace of a square matrix $A$, we will first prove the following lemma.

\begin{lemma}
    Suppose $G \subseteq \GL_n(\bb{C})$ is a matrix group and let $A, A', B \in G$. Assume that (i) $A$ is diagonalizable and (ii) if $\lambda \in \bb{C}$ is an eigenvalue of $A$, then there is some eigenvalue $ \lambda' \in \bb{C}$ of $A'$ with
     \[ E_{\lambda}(A) \subseteq E_{\lambda'}(A'). \]
    If $A \in C_G(B)$, then $A' \in C_G(B)$.
 \end{lemma}

 \begin{proof}
Enumerate the eigenvalues of $A'$ by $\left\{\lambda_i' \right\}_i$. Since $A$ is diagonalizable, condition (ii) implies that, for each index $i$, there are eigenvalues $\left\{\lambda_{i k} \right\}_k$ of $A$ with
 $$E_{\lambda_{i}'} (A')= \bigoplus_k E_{\lambda_{i k}} (A).$$
This also implies that $A'$ is diagonalizable:
$$\bb{C}^n = \bigoplus_{\lambda} E_{\lambda}(A) \subseteq \bigoplus_{\lambda'} E_{\lambda'}(A') \subseteq \bb{C}^n.$$

If $B$ commutes with $A$, then $B$ preserves each $A$-eigenspace $E_{\lambda_{i k}} (A)$. Thus, $B$ preserves each $A'$-eigenspace $E_{\lambda_{i}'} (A')$. Since $A'$ is diagonalizable, this implies $B$ commutes with $A'$. 
 \end{proof}
 
\begin{proposition}\label{prop:SU}
For $n \geq 1$, the groups $\SU(n), \mathrm{Sp}(n)$ and $\U(n)$ have Property \ref{A}. 
\end{proposition}

\begin{proof}
Write $G$ for $\SU(n), \mathrm{Sp}(n)$ or $\U(n)$ and fix a finite set $S\subseteq G$. Let $A \in C_G(S)$. To prove the proposition, we will connect $A$ by a path in $C_G(S)$ to an element of $Z(G)$. Denote by $T \subseteq G$ the standard maximal torus, so $T$ consists of the diagonal matrices with complex entries (we are viewing $\mathrm{Sp}(n) \subseteq \U(2n)$ in the usual way). Then there is some $P \in G$ so that $P^{-1} AP = \text{ diag}(\lambda_1, ..., \lambda_1, \lambda_2,..., \lambda_r)$ lies in $T$, with $\lambda_i \in S^1$ appears $e_i$ times for some $e_i \geq 1$.
    If $r = 1$, then $A  =\lambda_1 I \in Z(G)$ is already central, and so we can take the constant path. Suppose $r > 1$. Since $T$ is connected we can find a path $\gamma: [0, 1] \to T$ from $(\lambda_1,\lambda_2, ..., \lambda_r)$ to $(1, 1,..., 1)$. Write the components of $\gamma$ as $\gamma = (\gamma_1,\gamma_2,  \ldots, \gamma_r)$. Then
    $$A(t) = P\text{ diag}(\gamma_1(t), \ldots, \gamma_1(t), \gamma_2(t) , \ldots, \gamma_r(t))P^{-1}$$
    defines a path $[0,1]\to G$; again, $\gamma_i(t)$ is repeated $e_i$ times. 
    By the previous lemma, $A(t) \in C_G(S)$. Since $A(1) \in Z(G)$, this is the desired path. 
\end{proof}

\subsection{Proof of Theorem \ref{thm:1}}\label{sec:Proof1}
Fix an enumeration $\V \cong \{1, \ldots, n\}$ of the vertex set, and thus identify $G^\V \cong G^n$. Let $x = (x_1, \ldots, x_n) \in \R(K, G) \subseteq G^n$. 

First we claim that for each $0 \leq r \leq n$, there is a path in $\R(K, G)$ from $x$ to $Z(G)^r \times G^{n - r}$. The case $r = 0$ is trivial. Working inductively, we assume there is a path from $x$ to some $y= (y_i)_i \in Z(G)^{r-1} \times G^{n - r + 1}$. Let $S_r = \left\{y_i \; \vert \; \{ i, r\} \in \E \right\}$ be the components of $y$ corresponding to the vertices adjacent to $r$. Thus, the $r$th component $y_r\in C_G(S_r)$ commutes with all elements of $S_r \subseteq G$. Since $G$ satisfies Property \ref{A}, there is some path $\gamma : [0,1] \to C_G(S_r)$ from $y_r$ to an element of $Z(G)$. Define
$$\eta(t) = (y_1, \ldots, y_{r-1}, \gamma(t), y_{r+1}, \ldots, y_n).
$$
Then $\eta(t) \in \R(K, G)$ since $\gamma(t)$ commutes with $y_j$ for all $j \in S_r$. As such, $ \eta$ is a path from $y$ to an element of $Z(G)^r \times G^{n-r}$. Concatenating with the path from $x$ to $y$, we get the claimed path from $x$ to $Z(G)^r \times G^{n - r}$. 

Applying the claim with $r = n$, we see that each element of $\R(K, G)$ can be connected by a path in $\R(K, G)$ to $Z(G)^n$. To finish the proof, we will show that each $z \in Z(G)^n$ can be connected by a path in $\R(K, G)$ to the trivial representation $(1, \ldots, 1)$ in $\R(K, G) \subseteq G^n$. Towards this end, fix $z = (z_1, \ldots, z_n) \in Z(G)^n$ and for $0 \leq k \leq n$ define $z^k = (z^k_1, \ldots, z^k_n) \in Z(G)^n$ to be the tuple with 
$$z^k_j = 1 \; \textrm{for $j \leq k$ \hspace{1cm} and \hspace{1cm}} z^k_j = z_j\; \textrm{for $j > k$},$$
so $z^0 = z$ and $z^n = (1, \ldots, 1)$. Since $G$ is connected, for each $0 \leq k \leq n-1$ we can find a path $\gamma_k : z^k \rightsquigarrow z^{k+1}$ that sends the $k$th component to $1 \in G$ and fixes the other
components. Since all other components are in the center of $G$, the $ k$th component 
commutes with the components of adjacent vertices, so $\gamma_k(t) \in \R(K, G)$.  
Thus the composition of the $\gamma_k$'s is a path from $z$ to $(1, \ldots, 1)$, as desired. \qed

\section{Groups without Property \ref{A}}\label{sec:GroupsWithoutA}

\subsection{Examples}

We begin with a general result.

\begin{proposition}\label{prop:SemiSimple}
    Let $G$ be a compact, connected, semisimple Lie group. If $G$ is not a product of $\SU(n)$'s and $\operatorname{Sp}(n)$'s, then $G$ does not have Property \ref{A}.
\end{proposition}

\begin{proof}
    Let $G_{\mathbb{C}}$ be the complexification of $G$. Then $G_{\mathbb{C}}$ is a connected complex Lie group (hence a complex algebraic group) that is not a product of $\SL(n,\mathbb{C})$'s and $\operatorname{Sp}(n,\mathbb{C})$'s. Since $G$ is semisimple, so too is $G_{\mathbb{C}}$ and so \cite[Thm. 5.10]{FL} implies the $G_{\mathbb{C}}$-character variety $\mathcal{X}(\mathbb{Z}^3,G_{\mathbb{C}})$ is not connected. This implies the $G$-character variety $\mathcal{X}(\mathbb{Z}^3,G)=\mathcal{R}(\mathbb{Z}^3,G)/G$ is not connected \cite[Thm. 1.1]{FL}. Hence, the representation variety $\mathcal{R}(\mathbb{Z}^r,G)$ is not connected and so the proposition follows from Theorem \ref{thm:1}.
\end{proof}

\begin{example}
    If $G= \SO(n)$ for $n \geq 3$, then $G$ does not have Property \ref{A}.  
\end{example}

\begin{example}
    If $G= \Spin(n)$ for $n \geq 3$, then $G$ does not have Property \ref{A}.  
\end{example}

\begin{example}
    If $G$ is a compact exceptional Lie group, then $G$ does not have Property \ref{A}.  
\end{example}

\begin{example}
    If $G$ is a compact, connected, semisimple group that is not simply-connected, then $G$ does not have Property \ref{A}.  
\end{example}

The next proposition can be used to analyze Lie groups beyond the scope of Proposition \ref{prop:SemiSimple}. 

\begin{proposition}\label{prop:NotPropA}
Let $H$ be a Lie group and assume there is a central element $c \in Z(H)$ that (i) is not in the identity component of $Z(H)$, and (ii) lies in the derived subgroup $[H, H]$. If $L$ is any Lie subgroup of $Z(H)$ containing $c$, then the quotient $H/L$ does not have Property \ref{A}. 
\end{proposition}

\begin{proof}
By assumption on $c$, we can find $g, h \in H/L$ having lifts $\widetilde{g}, \widetilde{h} \in H$ with commutator $[\widetilde{g},\widetilde{h}] = c $. Then $g$ and $h$ commute in $H/L$ and so $h \in C_{H/L}(g)$ lies in the centralizer of $g$. We will show that $h$ cannot be connected by a path in $C_{H/L}(g)$ to an element of the center of $H/L$. Indeed, if $\gamma: [0,1] \rightarrow C_{H/L}(g)$ were such a path, then we could lift it to a path $\widetilde{\gamma}:[0, 1] \rightarrow H$ with 
\begin{itemize}
\item $\widetilde{\gamma}(0) = \widetilde{h}$, 
\item $\widetilde{\gamma}(1) \in Z(H)$, and
\item $[\widetilde{g}, \widetilde{\gamma}(t)] \in L$, for all $t \in [0, 1]$. 
\end{itemize}
Then $t \mapsto [\widetilde{g}, \widetilde{\gamma}(t)]$ would be a path in $Z(H)$ from $c$ to the identity, which is contrary to our assumption on $c$. 
\end{proof}

\begin{example}
Here we use Proposition \ref{prop:NotPropA} to give a second proof that $\SO(n)$ does not have Property \ref{A} for $n \geq 3$. Take $H = \Spin(n)$. The center of $\Spin(n)$ is discrete and contains a preferred subgroup $L \cong \bb{Z}_2$ with
$$\Spin(n) / L \cong \SO(n).$$
Let $c \in L$ be the non-trivial element of $L$; this is not in the identity component of $Z(\Spin(n))$. Being a simple group, $\Spin(n)$ is perfect and so it equals its derived subgroup; thus, $c \in [\Spin(n), \Spin(n)]$. It follows from Proposition \ref{prop:NotPropA} that $\SO(n)$ does not have Property \ref{A}. 
\end{example}

\begin{example}
Essentially the same argument shows that the projective linear group 
$$ \PSL(n, \bb{C}) = \SL(n, \mathbb{C}) / Z(\SL(n, \mathbb{C})) $$ 
does not have Property \ref{A}, provided $n \geq 2$. The key features are that (i) the center $Z(\SL(n, \mathbb{C})) \cong \mathbb{Z}_n$ is not connected, and (ii) $\SL(n, \mathbb{C})$ is perfect. 
\end{example}


\subsection{The obstruction map}

Throughout this section we assume that $G$ is a connected topological group that admits a universal cover $\widetilde{G}$, in the following sense: $\widetilde{G}$ is a topological group that is simply-connected and equipped with a covering map $\widetilde{G} \rightarrow G$ satisfying the universal property for universal covering spaces. By the general theory for covering spaces, if $G$ is locally path-connected and semilocally simply-connected, then $G$ admits such a cover. All Lie groups, being locally modeled on $\bb{R}^n$, admit universal covers. 

Fix a finite simple graph $K$ and consider the \emph{commutator map}
$$\mu_{K, G} : G^\V \longrightarrow G^\E, \hspace{1cm} (x_v)_{v \in \V } \longmapsto ([x_{v(e)}, x_{w(e)}] )_{e \in \E}.$$
Technically, this depends on a choice of orientation for each edge, which we assume has been made; we will write $v(e)$ for the initial vertex of $e$ and $w(e)$ for the terminal vertex. Our main results are not sensitive to how exactly this choice was made: When $G = \SO(3)$ all elements of $\pi_1(G)$ are involutions, in which case the map $\mu$ is independent of the choice of edge orientations. (From a group-theoretic standpoint, this choice of edge-orientation amounts to pinning down the relations in our presentation for $\Gamma_K$: do we want our presentation to have $[v(e), w(e)]$ or $[w(e), v(e)]$?)

Note that the representation variety $\R(K , G) = \mu_{K, G}^{-1}({1}_G)$ is the inverse image, where ${1}_G : \E \rightarrow G$ sends everything to the identity in $G$.

\begin{lemma}\label{lem:WellDefined}
    The value of $\mu_{K, \widetilde{G}}(\widetilde{x})$ is independent of choice of lift $\widetilde{x}$ of $x$.
\end{lemma}

\begin{proof}
    Let $\tilde{x}$, $\tilde{x}'$ be distinct lifts of $x$. Then for all $v$, there is some $c_v \in \pi_1(G)$ so that $\tilde{x}'_v = \tilde{x}_v c_v$.
    Let $e \in \E$ be an edge. Then since $\pi_1(G) \subseteq Z(\tilde{G})$ consists of central elements, we have
    \begin{align*}
        \mu_{\pi, \tilde{G}}(\tilde{x}')_e &= [\tilde{x}'_{v(e)}, \tilde{x}'_{w(e)}] = [\tilde{x}_{v(e)}, \tilde{x}_{w(e)}] = \mu_{\pi,\tilde{G}}(\tilde{x})_e.
    \end{align*} 
\end{proof}

Due to the lemma, we can define a map $\overline{\mu}: \G^\V \rightarrow \widetilde{G}^\E$ by $\overline{\mu}(x) = \mu_{K,\tilde{G}}(\tilde{x})$ for any lift $\tilde{x}$ of $x$. This fits into a commutative diagram:

\[
\begin{tikzcd}
	\tilde{G}^\V & \tilde{G}^\E \\
	G^\V & G^\E
	\arrow["{\mu_{K,\tilde{G}}}", from=1-1, to=1-2]
	\arrow[from=1-1, to=2-1]
	\arrow[from=1-2, to=2-2]
	\arrow["{\overline{\mu}}", dashed, from=2-1, to=1-2]
	\arrow["{\mu_{K,G}}", from=2-1, to=2-2]
\end{tikzcd}
\]

\begin{lemma}
    The map $\overline{\mu}$ is continuous.
\end{lemma}

\begin{proof}
    Consider an evenly covered
    neighborhood $U$, so that there is a continuous section
    $s : U \to \tilde{G}^\V$. Then $\overline{\mu}|_{U}$ is
    the composition $\mu_{\tilde{G}} \circ s$, so $\overline{\mu}|_U$ is continuous. Since the evenly
    covered neighborhoods of $G^\V$ cover $G^\V$, $\overline{\mu}$ is continuous.
\end{proof}

Let $x \in \R(K, G)$. Then $\mu_{K, G}(x)$ maps to the identity in $G^\E$. By the exactness of $\pi_1(G) \to \tilde{G} \to G$, it follows that $\mu_{K, G}(x)$ lifts to an element of $\pi_1(G)$. That is, if $x \in \R(K, G)$, then $\overline{\mu}(x) \in \pi_1(G)^\E$, and so the following is well-defined. 

\begin{definition}
The map $o_K : \R(K, G) \to \pi_1(G)^\E$ defined by $x \mapsto \overline{\mu}(x)$ is the \emph{obstruction map}.
\end{definition}

The following can be viewed as justifying the use of the word ``obstruction'' here.

\begin{proposition}\label{prop:HomExtension}
Let $x \in \R(\pi, G)$. Then $o_K(x) = 1_{\widetilde{G}}$ if and only if $\widetilde{x} : \V \rightarrow \widetilde{G}$ extends to a group homomorphism $\widetilde{x} \in \R(\pi, \widetilde{G})$.
\end{proposition}

\begin{proof}
The map $\widetilde{x}$ extends to a group homomorphism if and only if $\widetilde{x}(v(e))$ and $\widetilde{x}(w(e))$ commute for all $e \in \E$, which is the meaning of $o_K(x) = 1_{\widetilde{G}}$. 
\end{proof}

 The obstruction map $o_K$ is continuous and locally constant. As such, it descends to a map 
 $$\pi_0(\R(K, G)) \longrightarrow \pi_1(G)^\E.$$ 
 If this map on $\pi_0$ is injective, then $\R(K, G)$ has at most $\vert \pi_1(G) \vert^{\vert \E\vert}$ connected components; if it is surjective, then $\R(K, G)$ has at least $\vert \pi_1(G) \vert^{\vert \E\vert}$ connected components.

\begin{remark}\label{rem:ObExtension}
Let $\pi = \langle S \vert R \rangle$ be a finitely-presented group with generating set $S$ and relations $R$. Suppose all words making up $R$ happen to be homogeneous of degree zero in all variables (an example of this is the commutator $[a, b] = a b a^{-1} b^{-1}$). Then the same construction given above extends to give an ``obstruction map'' $o_\pi: \R(\pi, G) \rightarrow \pi_1(G)^{R}$; the homogeneity condition is needed to extend the proof of Lemma \ref{lem:WellDefined} to this situation. This too has the property that $x \in \R(\pi, G)$ lifts to an element $\widetilde{x} \in \R(\pi, \widetilde{G})$ if and only if $o_\pi(x) = 1_{\widetilde{G}}$. However, not every group $\pi$ has a presentation admitting only degree-zero homogeneous relations, and for such a group a different approach is needed to detect the when elements of $\R(\pi, G)$ lift; we give such an alternative approach in Section \ref{sec:ASW}. 
\end{remark}

The following will help us give a graph-theoretic interpretation of the obstruction map.

\begin{definition}
     A \emph{$\pi_1(G)$-edge marking} of a graph $K = (\V, \E)$ is any map $\Lambda \in \pi_1(G)^\E$ from $\E$ to $\pi_1(G)$. A \emph{$\pi_1(G)$-marked graph} is a pair $(K, \Lambda)$ where $K$ is a graph and $\Lambda$ is a $\pi_1(G)$-edge marking of $K$. For $c \in \pi_1(G)$, an edge $e$ is called $c$-\emph{marked} if $\Lambda_{e} = c$. 
 \end{definition}
\noindent When $G$ is clear from context, we drop ``$\pi_1(G)$'' from the notation. 

If $x\in \R(K, G)$, then $o_K(x)$ is an edge marking of $K$. We can thus categorize elements of $\R(K, G)$ in terms of edge markings. Indeed, if $\Lambda$ is an edge marking, then $o_K^{-1}(\Lambda)$ is either empty or a union of connected components of $\R(K, G)$. It is through the (disjoint) union 
\begin{equation}\label{eq:ODecomp}
\R(K, G) = \bigcup_{\Lambda \in \pi_1(G)^\E} o_K^{-1}(\Lambda)
\end{equation}
that we will analyze the components of $\R(K, G)$.

\begin{proposition}[Naturality of $o$]\label{prop:natty}
    Let $f : K' \to K$ be a graph homomorphism. 
    Then the following diagram commutes:
\[\begin{tikzcd}
	{\R(K, G)} & {\pi_1(G)^{\E}} \\
	{\R(K', G)} & {\pi_1(G)^{\E'}}
	\arrow["{o_{K}}", from=1-1, to=1-2]
	\arrow["{f^*}", from=1-1, to=2-1]
	\arrow["{f^*}", from=1-2, to=2-2]
	\arrow["{o_{K'}}", from=2-1, to=2-2]
\end{tikzcd}\]
\end{proposition}

\begin{proof}
    Let $e \in \E'$. Setting $v = v(e)$ and $w = w(e)$, for $x \in \R(K, G)$ we have
    \begin{align*}
        \big(f^*o_{K}(x)\big)_{e} = o_{K}(x)_{f(e)} = [\widetilde{x}_{f(v)}, \widetilde{x}_{f(w)}] = [(\widetilde{f^*x})_{v}, (\widetilde{f^*x})_{w}]= \big(o_{K'}(f^*x)\big)_{e}.
    \end{align*}
\end{proof}

 \begin{corollary}[Functoriality of $o^{-1}$]\label{cor:Funct}
     Let $f: K' \rightarrow K$ be a graph homomorphism and $\Lambda \in \pi_1(G)^{\E}$ an edge marking. Then the map $f^* : \R(K, G) \to \R(K', G)$ restricts to a well-defined map
     $f^* : o_K^{-1}(\Lambda) \to o_{K'}^{-1}(f^*\Lambda)$.
 \end{corollary}

\subsection{Bundles and flat connections}\label{sec:Bundles}

Let $G$ be a compact Lie group and suppose $\pi = \langle R \vert S \rangle$ is finitely-presented. Here we give a topological interpretation of $\R(\pi, G)$. 

Let $X$ be any manifold having fundamental group $\pi$. By a result of Dehn \cite{Dehn} such a manifold always exists. 
Suppose $P \rightarrow X$ is a principal $G$-bundle, and write $\A_{\mathrm{flat}}(P)$ for the space of flat connections on $P$; see \cite{KN} for a general overview of these matters. Fix a basepoint $x_0 \in X$ and denote by $\G_0(P)$ the group of bundle isomorphisms of $P$ covering the identity on $X$ and that act as the identity on the fiber over $x_0$. This fits into a short exact sequence
\begin{equation}\label{eq:Gauge}
1 \longrightarrow \G_0(P) \longrightarrow \G(P) {\longrightarrow} G \longrightarrow 1
\end{equation}
where $\G(P)$ is the group of all bundle isomorphisms covering the identity. The holonomy for connections induces an embedding
\begin{equation}\label{eq:Connections0}
\iota_P: \A_{\mathrm{flat}}(P) / \G_0(P) \longhookrightarrow \R(\pi, G).
\end{equation}
Allowing $P$ to run over all isomorphism types $[P]$ of principal $G$-bundles on $X$ gives a homeomorphism
\begin{equation}\label{eq:Connections}
\R(\pi, G) \cong \bigcup_{[P] } \A_{\mathrm{flat}}(P) / \G_0(P).
\end{equation}
Indeed, if $x \in \R(\pi, G)$ is a representation, then we can construct $P$ and a flat connection $A$ with holonomy $x$ as follows: Let $\widetilde{X}$ be the universal cover of $X$, and consider the action of $\pi \cong \pi_1(X, x_0)$ on $\widetilde{X} \times G$ with $\pi_1(X, x_0)$ acting on $\widetilde{X}$ by deck transformations and on $G$ by $x$. Then the quotient
$$P = (\widetilde{X} \times G )/ \pi_1(X,x_0)$$
is naturally a principal $G$-bundle on $X$, and the trivial connection on $\widetilde{X} \times G \rightarrow \widetilde{X}$ descends to give the desired flat connection on $P$.

To simplify the discussion moving forward, we restrict attention the case where $\pi = \Gamma_K = \langle \V \vert \E \rangle$ is a RAAG. 

\begin{lemma} \label{lem:FlatCon}
Assume $G$ is connected and $X$ is a manifold with $\pi_1(X) \cong \Gamma_K$. 

\begin{enumerate}
\item[(a)] If $x \in \R(K, G)$, then there is a bundle $P_x$ so that $x$ lies in the image of (\ref{eq:Connections0}) (with $P_x$ in place of $P$). This bundle $P_x$ is unique up to bundle isomorphism. 

\item[(b)] The isomorphism type of $P_x$ depends on $x$ only through the value of $o_K(x)$. That is, for $x, y \in \R(K, G)$, the bundles $P_x$ and $P_y$ are isomorphic if and only if $o_K(x) = o_K(y)$.
\end{enumerate}
\end{lemma}

\begin{proof}
Given $x \in \R(K, G)$, we will construct the bundle $P_x$ directly; the claim in (b) will be clear from the construction. In brief, bundles on $X$ supporting flat connections are entirely determined by their restriction to the 2-skeleton, and this is encoded by $o_K(x)$. 

In more detail, begin by declaring $P_x$ to be trivial over the 1-skeleton $X_1$ of $X$; this is our only option since $G$ is connected. Note that the 1-skeleton can be explicitly constructed from the generating set $\V$: Since $\pi_1(X) \cong \Gamma_K = \langle \V \vert \E \rangle$, for each $v \in \V$ we can find a based loop $\gamma_v$ in $X$, and these $\gamma_v$'s for $v \in \V$ give $X_1$.

The construction of $P_x$ on the 2-skeleton is predicated on the following observation. Suppose $\beta: \bb{D}^2 \rightarrow X$ is a 2-cell that restricts to the boundary to be a map $\alpha: \partial \bb{D}^2 = S^1 \rightarrow X_1$ into the 1-skeleton. The topological type of any principal $G$-bundle $P$ over this 2-cell is uniquely determined as follows: The pullback $\beta^* P$ is a bundle over $\bb{D}^2$ and so is trivializable. Any section of $\beta^* P$ restricts to the boundary of $\bb{D}^2$ to yield a map $S^1 \rightarrow G$, the homotopy type of which determines $P \vert_\beta$ uniquely, up to isomorphism. 

With this understood, let $e \in \E$ be an edge with vertices $v, w$. Since $\pi_1(X) \cong \Gamma_K$, we can find based loops $\gamma_v, \gamma_w$ in $X_1$ with homotopy class $v, w$, respectively, and with the property that $[\gamma_v, \gamma_w]$ is contractible in $X$. Thus, we can find a map $\beta_e: \bb{D}^2 \rightarrow X$ that restricts to $[\gamma_v, \gamma_w]$ on the boundary. Define $P$ on this 2-cell $\beta_e$ by using $o_K(x)_e \in \pi_1(G)$ as an attaching map. Repeat for all $e \in \E$. We claim this uniquely defines $P_x$ on the 2-skeleton of $X$. To see this, suppose $\beta$ is any other 2-cell with boundary mapping into $X_1$. Then we can write the homotopy class of its boundary as a product of elements of $\V$. Then $\beta$ itself represents some word in these elements, which is in the normal subgroup generated by the words of $\E$. To be precise, the relations are viewed as elements in the free group $Fr(\V)$ generated by $\V$, in which $\beta$ is a product of formal elements of the form $u [v, w] u^{-1}$ where $u, v, w \in Fr(\V)$ and $v, w \in \V$ are the endpoints of some edge $e$. The attaching map for $\beta$ is then given by the product of 
$$\widetilde{x}(u) [\widetilde{x}(v), \widetilde{x}(w)] \widetilde{x}(u)^{-1} = o_K(e),$$
since $ [\widetilde{x}(v), \widetilde{x}(w)] = o_K(e)$ is central in $\widetilde{G}$. 

Finally, on any cells of dimension $k \geq 3$, define $P_x$ by taking the attaching map $\partial \bb{D}^k \rightarrow G$ to be homotopically trivial; when this is the case for all $k$-cells in $X$, we will say the bundle \emph{trivializes over the $k$-cells}. This establishes the uniqueness of $P_x$ as well since if $P$ is any bundle that does not trivialize over the $k$-cells for some $k \geq 3$, then $P$ does not admit any flat connections. 
\end{proof}

We see from this construction the following: If the restriction $P \vert_{X_2}$ to the $2$-skeleton $X_2$ admits a flat connection, and if $P$ trivializes over the $k$-cells for $k \geq 3$, then $\A_{\mathrm{flat}}(P)$ is not empty. In fact, since $\pi_2(G) = 0$ for all Lie groups $G$, every principal $G$-bundle trivializes over the 3-cells. Thus, we have:

\begin{corollary}\label{cor:Surj}
Assume $X$ is a 3-manifold with $\pi_1(X)$ a RAAG. If $o_K: \R(K, G) \rightarrow \pi_1(G)^\E$ is surjective, then every principal $G$-bundle on $X$ admits a flat connection.
\end{corollary}

For future reference, we also encode the following consequence of Lemma \ref{lem:FlatCon}.

\begin{corollary}\label{cor:Components}
Let $X$ be any manifold with $\pi_1(X) = \Gamma_K$ a RAAG. 
\begin{enumerate}
\item[(a)] Let $\Lambda \in \pi_1(G)^\V$. If $o_K^{-1}(\Lambda)$ is nonempty, then there is a principal $G$-bundle $P \rightarrow X$, unique up to isomorphism, so that the map $\iota_{P}$ of (\ref{eq:Connections0}) induces a homeomorphism
$$\A_{\mathrm{flat}}(P)/\G_0(P) \cong o_K^{-1}(\Lambda).$$

\item[(b)] Let $P \rightarrow X$ be a principal $G$-bundle. If $\A_{\mathrm{flat}}(P)$ is not empty, then there is a unique $\Lambda \in \pi_1(G)^\E$ so that $\iota_{P}$ of (\ref{eq:Connections0}) induces a homeomorphism
$$\A_{\mathrm{flat}}(P)/\G_0(P) \cong o_K^{-1}(\Lambda).$$
\end{enumerate}
\end{corollary}

\noindent In this way, the decomposition (\ref{eq:Connections}) in terms of bundle types corresponds precisely to the decomposition (\ref{eq:ODecomp}) in terms of fibers of the obstruction map.

\subsection{Digression: The second Stiefel--Whitney class}\label{sec:ASW}

 Here we give an interpretation of the obstruction map $o_K$ in terms of the second Stiefel--Whitney class. The material of this section is not used elsewhere in this paper.

 Fix a finitely-presented group $\pi$. Let $G$ be a topological group admitting a universal cover and suppose, in addition, that $\pi_1(G)$ is discrete and abelian; such is the case when $G$ is a semisimple Lie group. Let $\pi$ be a finitely-presented group and view $\pi_1(G)$ as $\pi$-module with the trivial action of $\pi$. Then we can associate to this the group cohomology $H^k(\pi, \pi_1(G))$. We will ultimately only be interested in this for $k = 2$, but we recall that the $k$-cochains are precisely the maps $\pi^k \rightarrow \pi_1(G)$, and the boundary operator on $1$-cochains $f: \pi \rightarrow \pi_1(G)$ is given by 
  $$d f(\gamma_1 , \gamma_2) = f(\gamma_2) f(\gamma_1 \gamma_2)^{-1} f(\gamma_1)$$
  while its action on $2$-cochains $w: \pi^2 \rightarrow \pi_1(G)$ is 
  $$dw (\gamma_1, \gamma_2, \gamma_3 ) = w(\gamma_2, \gamma_3) w(\gamma_1 \gamma_2, \gamma_3)^{-1} w(\gamma_1, \gamma_2 \gamma_3) w(\gamma_1, \gamma_2)^{-1}.$$
Since we are electing to write the elements of $\pi_1(G) \subseteq \widetilde{G}$ multiplicatively, we do the same for cochains. In particular, we will write $1$ for the identity of $H^2(\pi, \pi_1(G))$.

Now we will associate to each $x \in \R(\pi, G)$ an element $w_x^G \in H^2(\pi, \pi_1(G))$; our approach mimics that of \cite[pp. 8--9]{DS}. Viewing $x$ as a map $\pi \rightarrow G$, let $\widetilde{x}: \pi \rightarrow \widetilde{G}$ be any lift (this lift need not be a group homomorphism). Define $\widetilde{w}_x : \pi^2 \rightarrow \widetilde{G}$ by
$$\widetilde{w}_x(\gamma_1, \gamma_2) = \widetilde{x}(\gamma_1 \gamma_2) \widetilde{x} (\gamma_2)^{-1} \widetilde{x}(\gamma_1)^{-1}.$$
Since $x$ is a group homomorphism, this descends to the identity in $G$ and so $\widetilde{w}_x$ can be viewed as a map $\pi^2 \rightarrow \pi_1(G)$. As one can check, $\widetilde{w}_x$ is closed and so descends to a group cohomology class
$$w_x \defeq [ \widetilde{w}_x] \in H^2(\pi, \pi_1(G)).$$
This is independent of the choice of lift $\widetilde{x}$. We will call $w_x$ the \emph{algebraic second Stiefel--Whitney class of $x$}.  

\begin{proposition}\label{prop:Lift}
Let $x \in \R(\pi, \pi_1(G))$. Then $w_x  = 1 \in H^2(\pi, \pi_1(G))$ if and only if $x$ has a lift $\pi \rightarrow \widetilde{G}$ that is a group homomorphism. 
\end{proposition}

\begin{proof}
If $w_x = 1$, then $\widetilde{w}_x = d f$ for some $f: \pi^2 \rightarrow \pi_1(G)$. It follows that $\widetilde{x} f^{-1}: \pi \rightarrow \pi_1(G)$ is another lift of $x$ and the condition $\widetilde{w}_x = df$ is exactly that $\widetilde{x} f^{-1}$ is a group homomorphism. The converse is clear. 
\end{proof}

Now suppose $\pi$ has a presentation $\langle S \vert R \rangle$ as in Remark \ref{rem:ObExtension} so that the obstruction map $o_\pi: \R(\pi, G) \rightarrow {\pi_1(G)}^R$ is defined. Then we see from Proposition \ref{prop:HomExtension} that if $x \in \R(\pi, G)$, then
$$w_x = 1 \Longleftrightarrow o_\pi(x) = 1_{G}.$$

We can say a bit more in the case where $\pi = \Gamma_K$ is a RAAG (so $S = \V$ and $R = \E$). Let $e \in \E$, and set $v = v(e)$ and $w = w(e)$. Thus, $vw = wv$ in $\Gamma_K$. Then we have
$$\begin{array}{rcl}
o_K(x)_e & = & \widetilde{x}(v) \widetilde{x}(w) \big( \widetilde{x}(w) \widetilde{x}(v) \big)^{-1}
\end{array}$$
Using $\widetilde{x}(v) \widetilde{x}(w) = \widetilde{w}_x(v, w)^{-1} \widetilde{x}(vw) $ and the analogous identity for $\widetilde{x}(w) \widetilde{x}(v)$, we get
\begin{equation}\label{eq:OSW}
o_K(x)_e =  \widetilde{w}_x(v, w)^{-1} \widetilde{w}_x(w, v)^{-1}.
\end{equation}

To tie this in with the Stiefel--Whitney class for bundles, consider the case $G = \SO(n)$ for $n \geq 3$. Then $\pi_1(G) \cong \left\{ \pm 1 \right\}$ consists of two elements, and we infer from (\ref{eq:OSW}) that 
\begin{equation}\label{eq:OKSW}
o_K(x)_e = 1\; \; \; \Longleftrightarrow \; \; \; \widetilde{w}_x(v, w) = \widetilde{w}_x(w, v)
\end{equation}
for all $e = \left\{v, w \right\} \in \E$. Now fix $x \in \R(\pi, \pi_1(G))$. By (\ref{eq:Connections}), there is some principal $G$-bundle $P \rightarrow X$ so that $x$ is in the image of (\ref{eq:Connections0}). As we saw in Proposition \ref{prop:Lift}, the group cohomology class $w_x^G \in H^2(\pi, \pi_1(G))$ is the obstruction to $x$ lifting to a representation $\pi \rightarrow \widetilde{G}$. When $x$ does lift, the structure group of $P$ reduces to $\widetilde{G}$. The second Stiefel--Whitney class $w_2(P) \in H^2(X, \bb{Z}/2\bb{Z})$ is also precisely this obstruction, and
$$w_x^G = w_2(P)$$
under the identification $H^2(\pi, \pi_1(G)) \cong H^2(X, \bb{Z}/ 2 \bb{Z})$; see \cite[p. 9]{DS}. Combining this with (\ref{eq:OKSW}) gives a relationship between the obstruction map $o_K$ and the second Stiefel--Whitney class $w_2$ when $\pi = \Gamma_K$ is a RAAG.

\section{Special Case: \texorpdfstring{$G = \SO(3)$}{}}\label{sec:SO(3)}

Our ultimate aim is to prove Theorems \ref{thm:2} and \ref{thm:3} as well as Corollary \ref{cor:4}; we carry this out in Section \ref{sec:Proofs} after a careful analysis of the obstruction map $o_K$ in Section \ref{sec:Analysis} and a review of quaternions.

\subsection{Quaternionic preliminaries}

Recall that if $q = x_0 + x_1 i + x_2 j + x_3 k$ is a quaternion with $x_0, \ldots, x_3 \in \bb{R}$, then the \emph{quaternionic conjugate} of $q$ is 
$$q^* \defeq x_0 - x_1 i - x_2 j - x_3 k$$
and its \emph{real} and \emph{imaginary} parts of $q$ are
$$\mathrm{Re}(q) = \frac{1}{2}(q + q^*) = x_0, \hspace{1cm} \mathrm{Im}(q) = \frac{1}{2}(q - q^*) =  x_1 i + x_2 j + x_3 k,$$
respectively. Write $\mathrm{Im}(\bb{H})$ for the subspace of purely imaginary quaternions (i.e., those $q$ with $q = -q^*$). This is a 3-dimensional vector space that we identify with $\bb{R}^3$. As such, we will view $S^2\subset S^3$ as the purely imaginary unit length quaternions
$$S^2= S^3 \cap \mathrm{Im}(\bb{H}).$$

Note the formula
$$\langle a, b \rangle = \mathrm{Re}(a b^*)$$
recovers the standard inner product under the obvious identification $\bb{H} \cong \bb{R}^4$. Moreover, if $q, p \in S^3$, then
$$\langle q a p, q b p \rangle = \langle a, b \rangle$$
since $q q^* = p p^* =  1$. Restricting to the case where $p = q^{-1}$, we see that $S^3$ acts on $\mathrm{Im}(\bb{H}) \cong \bb{R}^3$ by conjugation, and this action preserves the inner product. It is in this way that we get a map
$$S^3 \longrightarrow \SO(3), \hspace{1cm} q \longmapsto (a \mapsto q a q^{-1}).$$
This is a covering map and it is through this that we view $S^3$ as the universal cover of $\SO(3)$. 

Throughout this section, for $x\in \SO(3)$, the symbol $\tilde{x}\in S^3$ will denote a lift of $x$. We identify $\pi_1(G)\cong \{\pm 1\}\subseteq S^3$. As we have seen in our discussion of the obstruction map, if $x,y \in \SO(3)$ commute, then these have lifts $\widetilde{x}$ and $\widetilde{y}$ with $[\widetilde{x}, \widetilde{y}] \in \left\{ \pm 1 \right\}$. The next lemmas give geometric interpretations for when the commutator is $+1$ and $-1$, respectively.

\begin{lemma}\label{lem:Com}
    Let $a, b \in S^2 \subseteq S^3$. Then $[a, b]  = 1$ $\iff$ $a = \pm b$.
\end{lemma}

\begin{proof}
One direction is trivial. For the converse, assume $[a, b] = 1$ for $a, b \in S^2$. We will first prove that $a  =\pm b$ under the assumption that $a = i$. Write $b = b_1 i + b_2 j + b_3 k$. A direct computation using $[i, j] = [i, k] =-1$ shows that when $a = i$ we must have $b_2 = b_3 = 0$. It follows that $b = b_1 i = \pm i$. 

For general $a \in S^2$, use the fact that $\SO(3)$ acts transitively on $S^2$: There is some $q \in S^3$ so that $q a q^{-1} = i$. Then $q a q^{-1}$ and $q b q^{-1}$ commute, so the argument of the previous paragraph shows that $q b q^{-1} = \pm i$; thus $b = \pm a$. 
\end{proof}

\begin{lemma}\label{lem:AntiCom}
    Let $a, b \in S^3$. Then $[a, b] = -1$ $\iff$
    $a, b \in S^2$ and $a \perp b$.
\end{lemma}
\begin{proof}
If $[a, b] = -1$, then $- b = ab a^{-1}$. Taking the real part we find
$$-\mathrm{Re}(b) = \mathrm{Re}(aba^{-1}) = \mathrm{Re}(b) .$$
It follows that $b \in \mathrm{Im}(\bb{H})$ is purely imaginary. A similar argument shows that $\mathrm{Re}(a) = \mathrm{Re}(ab) = 0$. The first tells us that $a \in \mathrm{Im}(\bb{H})$, while the second tells us that $a$ and $b$ are othogonal: We have $b^* = -b$ and so 
$$\langle a, b \rangle = \mathrm{Re}(ab^*) =  - \mathrm{Re}(ab)  = 0.$$

Conversely, suppose $a, b \in S^2$ with $a, b$ perpendicular. The group $\SO(3)$ acts transitively on the orthonormal pairs in $S^2$, so there is some $q \in S^3$ with $a = qi q^{-1}$ and $b = q j q^{-1}$. This gives
$$[a, b] = [q i q^{-1}, qj q^{-1}] = q[ i, j] q^{-1}= q (-1) q^{-1} =  -1.$$
\end{proof}

\subsection{Analysis of the obstruction map}\label{sec:Analysis}

Fix a graph $K = (\V, \E)$. The obstruction map $o_K$ is a map of the form $\R(K, \SO(3)) \rightarrow \left\{ \pm 1 \right\}^\E$. As such, relative to any edge marking, each edge will be labeled with either $+1$ or $-1$. The main results of this section are as follows.

\begin{theorem}\label{thm:Inj}
    Let $K$ be a disjoint union of cycles, trees, and complete graphs. Then the map 
    $$\pi_0(\R(K, \SO(3))) \longrightarrow \left\{ \pm 1 \right\}^{\E}$$
    induced by the obstruction map $o_K$ is injective. In particular, 
    $$\vert \pi_0(\R(K, \SO(3))) \vert \leq 2^{\vert \E \vert}.$$
\end{theorem}

\begin{theorem}\label{thm:Surj}
    Let $K$ be a disjoint union of cycles and trees. Then the map 
    $$ \pi_0(\R(K, \SO(3))) \longrightarrow \left\{ \pm 1 \right\}^{\E}$$
    induced by the obstruction map $o_K$ is surjective. In particular,
    $$\vert \pi_0(\R(K, \SO(3))) \vert \geq 2^{\vert \E \vert}.$$
    
    Conversely, if $K$ is a graph with the property that $o_K$ is surjective, and if $K_0 \subseteq K$ is a component that is not a cyclic graph, then $K_0$ contains no 3-cycles.
\end{theorem}

We prove these below, after we develop a better understanding of the fibers of the obstruction map. Before getting to that, we highlight the following two examples showing that injectivity and surjectivity can each fail.

\begin{example}
    Here we show the map $\pi_0(\R(K, \SO(3))) \rightarrow \left\{ \pm 1 \right\}^{\E}$ need not be injective. Consider the marked graph $(K, \Lambda)$ illustrated here:
    
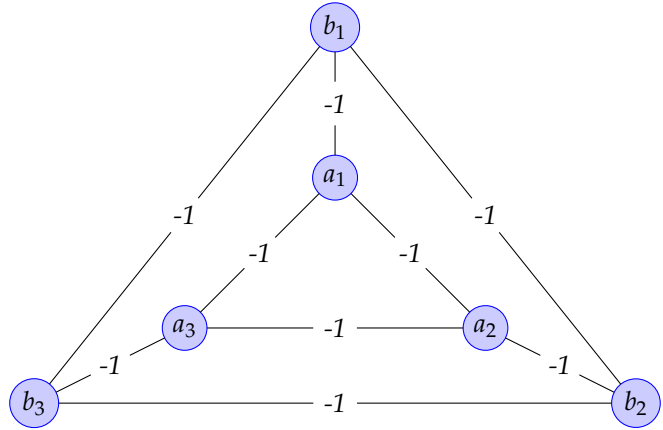
\begin{center}
        \begin{tikzpicture}
   [every node/.style={fill=white,circle,minimum width=0pt,inner sep=2pt, outer sep =0pt}, dot/.style={circle,draw=blue,fill=blue!20!}]
            \node[dot] (A1) at (0, 0) {$a_1$};
            \node[dot] (A2) at (2, -2) {$a_2$};
            \node[dot] (A3) at (-2, -2) {$a_3$};

            \node[dot] (B1) at (0, 2) {$b_1$};
            \node[dot] (B2) at (4, -3) {$b_2$};
            \node[dot] (B3) at (-4, -3) {$b_3$};

            \draw (A1) edge node {-1} (A2);
            \draw (A2) edge node {-1} (A3);
            \draw (A3) edge node {-1} (A1);
            \draw (B1) edge node {-1} (B2);
            \draw (B2) edge node {-1} (B3);
            \draw (B3) edge node {-1} (B1);
            \draw (B1) edge node {-1} (A1);
            \draw (B2) edge node {-1} (A2);
            \draw (B3) edge node {-1} (A3);
            \end{tikzpicture}
    \end{center}
    Suppose $(a_1, a_2, a_3, b_1, b_2,b_3) \in o^{-1}_K(\Lambda)$.
    \begin{enumerate}
        \item If $a_2 = b_1$, then $a_3 \neq a_2 = b_1$. Since
        $\widetilde{a}_1$ and $\widetilde{b}_3$ are both perpendicular to both $\widetilde{a}_3$ and $\widetilde{b}_1$, it follows that $a_1 = b_3$. Likewise, $a_3 = b_2$.
        \item  If $a_2 \neq  b_1$, then $a_1 = b_2$, $a_2 = b_3$, and $a_3 = b_1$.
    \end{enumerate}
    So the fiber $o^{-1}(\Lambda)$ has two components: one with $(a_1, a_2, a_3) = (b_3, b_1, b_2)$ and one with $(a_1, a_2, a_3) = (b_2, b_3, b_1)$.

    A similar construction yields marked graphs with $\pi_0(\R(K, \SO(3))) \rightarrow \left\{ \pm1 \right\}^{\E}$ having $2^n$-component fibers.
\end{example}

\begin{example}\label{ex:1}
Here we show the map $\pi_0(\R(K, \SO(3))) \rightarrow \left\{ \pm 1 \right\}^{\E}$ need not be surjective. Consider the following marked graph $(L,\Lambda)$ with vertices labeled $a, b, c, d$:
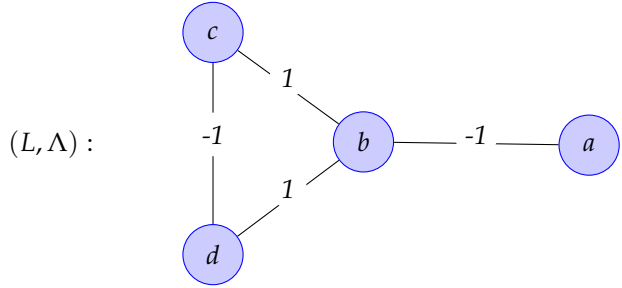
\begin{center}
    \vspace{4pt}
  $(L, \Lambda):$ \hspace{1cm}  \begin{tikzcd}
        \node  (A) at (3, 0)[circle,draw=blue,fill=blue!20!,minimum size=.8cm]{a};
        \node  (B) at (0, 0)[circle,draw=blue,fill=blue!20!,minimum size=.8cm]{b};
        \node  (C) at (-2, 1.5)[circle,draw=blue,fill=blue!20!,minimum size=.8cm]{c};
        \node  (D) at (-2, -1.5)[circle,draw=blue,fill=blue!20!,minimum size=.8cm]{d};
\begin{scope}[every node/.style={fill=white,circle,minimum width=0pt,inner sep=2pt, outer sep =0pt}]
        \draw (B) edge node {1} (D);
        \draw (D) edge node {-1} (C);
        \draw (C) edge node {1} (B);
        \draw (B) edge node {-1} (A);
        \end{scope}
    \end{tikzcd} 
    \end{center}
    This edge marking $\Lambda$ is not in the image of $o_L$: If $\Lambda = o_L(x)$ for some $x$, then $x(b)$ and $x(c)$ would admit lifts $\widetilde{x}(b), \widetilde{x}(c) \in S^3$ that commute; likewise, $\widetilde{x}(b), \widetilde{x}(d)$ commute. Since $b$ is incident to an edge labeled $-1$, it follows that $\widetilde{x}(b)$ is not in the center of $S^3$. Centralizers of non-central elements in $S^3$ are abelian. Since $\widetilde{x}(c), \widetilde{x}(d) \in C_{\SO(3)}(\widetilde{x}(b))$, it follows that $\widetilde{x}(c)$ and $\widetilde{x}(d)$ commute, contrary to the label of $-1$ on the edge connecting $c$ and $d$.  
\end{example}

Our strategy for detecting the components of $\R(K, \SO(3))$ is to consider its behavior under vertex-deletion and edge-contraction in $K$. We thus need an understanding of how the fibers of the obstruction map behave under these operations. In general this is complicated, but the situation simplifies when the edge markings are sufficiently controlled.

We begin with vertex-deletion. Given a vertex $v \in \V$, write $K - \left\{v\right\}$ for the graph obtained by deleting $v$ and all edges incident to $v$. Write $N_v \subseteq \V$ for those vertices adjacent in $K$ to $v$.

\begin{lemma}[Vertex-Deletion]\label{lem:VertexDeletion}
    Let $\Lambda$ be an edge-marking for $K$. Assume that $v \in \V$ is a vertex with the property that $\Lambda_{\left\{v, w \right\}} = 1$ for all $w \in N_v$. Set $K' = K - \left\{v\right\}$ and let $i : K' \to K$ denote the inclusion. Then pullback $i^*: \R(K, \SO(3)) \rightarrow \R(K', \SO(3))$ restricts to a well-defined surjection
    \[ i^* : o_K^{-1}(\Lambda) \longrightarrow  o^{-1}_{K'}(i^*\Lambda)\]
    that induces a bijection on $\pi_0$.
\end{lemma}

\begin{proof}
    That the restriction of $i^*$ is well-defined is just Corollary \ref{cor:Funct}.
    To see it is surjective, let $x' \in o^{-1}_{K'}(i^* \Lambda)$. This can be viewed as prescribing labels for the vertices of $K'$ by elements of $\SO(3)$, in such a way that adjacent vertices have $S^3$ lifts that either commute or anticommute, according to the markings given by $i^* \Lambda$. Passing now to $K$, since all edges incident to $v$ are 1-labeled, there is a canonical element $x_0$ of the fiber $(i^*)^{-1}(x')$ given by 
    $$x_0(v) = 1, \hspace{1cm} x_0(w) = x'(w)\; \textrm{for $w \neq v$}.$$

    Fix $x \in (i^*)^{-1}(x')$; we will connect $x$ to $x_0$ by a path in $(i^*)^{-1}(x')$. Since $S^3 \cong \SU(2)$, it follows from Proposition \ref{prop:SU} that $S^3$ has Property \ref{A}. In particular, there is a path $\gamma : \widetilde{x}_v \rightsquigarrow 1$ within the image of $\bigcap_{w \in N_v} C_{S^3}(\widetilde{x}_w) \subseteq S^3$. Let $\eta$ be the path in $\SO(3)^\V$ that equals $\gamma$ on $v$ and fixes all other components. Then $\eta$ is the path we are after.

    We have just seen that $i^*$ is surjective with connected fibers. It is also proper since $K$ is finite and $\SO(3)$ is compact, so it follows that $i^*$ induces a bijection $\pi_0(o_K^{-1}(\Lambda)) \cong \pi_0(o_{K'}^{-1}(i^* \Lambda))$. 
\end{proof}

 Given an edge $e$, let $K''  = K / e$ be the graph obtained from $K$ by shrinking $e$ to a point and identifying to one edge any multi-edges created in this process. Let $q : K \rightarrow K''$ denote the quotient map.

\begin{lemma}[Edge-Contraction]\label{lem:EdgeContraction}

Let $\Lambda$ be an edge marking for $K$. Assume $e \in \E$ is a $1$-marked edge, and either (i) at least one of $v(e)$ or $w(e)$ is a leaf, or (ii) $v(e)$ and $v(w)$ are each incident to some $-1$-marked edge. Set $K''  = K / e$.

Assume there is an edge marking $\Lambda''$ for $K''$ so that $\Lambda = q^* \Lambda''$. Then $\Lambda''$ is unique and the pullback map $q^*: \R(K'', \SO(3)) \rightarrow \R(K, \SO(3))$ restricts to a homeomorphism
            \[ q^* : o^{-1}_{K''}(\Lambda'') \longrightarrow o^{-1}_K(\Lambda).\]
\end{lemma}

\begin{proof}
    Any continuous bijection between compact Hausdorff spaces is a homeomorphism, so it suffices to show that $q^*$ is bijective. Injectivity of $q^*$ (even on the bigger space $\R(K'', \SO(3))$) is immediate from the surjectivity of $q$. To see that $q^*$ is surjective on the inverse images, let $x \in o_K^{-1}(\Lambda)$. Define a map $x'' : \V'' \rightarrow \SO(3)$ as follows: Let $v'' \in \V''$. If $v'' \neq q(v(e))$, then there is a unique $v \in \V$ with $v'' = q(v)$ and we set $x''(v'') \defeq x(v)$. If $v''= q(v(e))$ (and thus also equals $q(w(e))$), then set $x''(v'') \defeq x(v(e))$ in case (ii) or if case (i) holds with $w(e)$ of valance 1. In case (i) where $w(e)$ does not have valance 1, then $v(e)$ does and we set $x''(v'') \defeq x(w(e))$. 
    
    In any case, we automatically have that $q^* x'' = x$, so it suffices to check that $x'' \in \R(K'', \SO(3))$ really is a representation (i.e., takes adjacent vertices to commuting elements). In case (i) this is immediate since the valancy condition implies there are no new relations on $x''(q(v(e)))$. 
    
    We may therefore assume case (ii), so each of $v(e)$ and $w(e)$ is the endpoint of some $-1$-marked edge. See the figure in Example \ref{ex:EdgeContraction}. It follows from Lemma \ref{lem:AntiCom} that $\widetilde{x}_{v(e)}, \widetilde{x}_{w(e)} \in S^2$. Then $[\widetilde{x}_{v(e)}, \widetilde{x}_{w(e)}] = 1$ implies $\widetilde{x}_{v(e)} = \pm \widetilde{x}_{w(e)}$. Thus $x_{v(e)} = x_{w(e)}$. Then any new relations are automatically satisfied by $x''(q(v(e)))$. 
    \end{proof}

In the hypotheses of  Lemma \ref{lem:EdgeContraction}
, we assumed that $\Lambda = q^* \Lambda''$ for some edge marking $\Lambda''$ for $K''$. When this is the case we say that $\Lambda''$ is an \emph{$e$-reduction of $\Lambda$}. Not every edge marking on $K$ has an $e$-reduction.

\begin{example}\label{ex:EdgeContraction}
Illustrated on the top in the figure below is a marked graph $(K, \Lambda)$ with some of the markings highlighted. The edge $e$ (unlabeled) connects vertices $v$ and $w$, and is 1-marked. Note that $v$ and $w$ are each incident to at least one $-1$-marked edge; this corresponds to case (ii) in Lemma \ref{lem:EdgeContraction}. The quotient map $q$ deletes the edge $e$ producing the graph $K''$ on the bottom with $v'' = q(v) = q(w)$ and $u'' = q(u)$. In Lemma \ref{lem:EdgeContraction} we assumed that $\Lambda$ has an $e$-reduction in the sense that $\Lambda =q^* \Lambda''$ is pulled back from a marking $\Lambda''$ on $K''$. The edge marking $\Lambda$ would not have an $e$-reduction if, for example, the edge $\left\{u, v \right\}$ in $K$ were 1-marked.
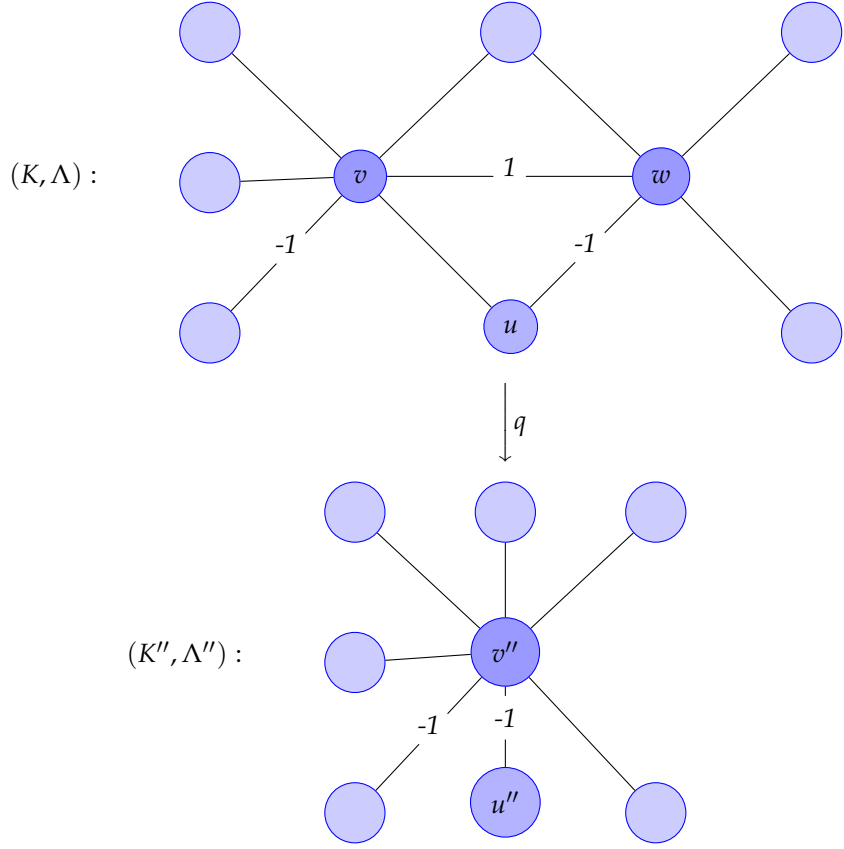
\begin{center}
    \vspace{4pt}
  $(K, \Lambda):$\hspace{1cm} \begin{tikzcd}
        \node  (TL) at (-4, 2)[circle,draw=blue,fill=blue!20!,minimum size=.8cm]{};
        \node  (TM) at (0, 2)[circle,draw=blue,fill=blue!20!,minimum size=.8cm]{};
        \node  (TR) at (4, 2)[circle,draw=blue,fill=blue!20!,minimum size=.8cm]{};
        \node  (ML) at (-4, 0)[circle,draw=blue,fill=blue!20!,minimum size=.8cm]{};
        \node  (MM) at (-2, 0)[circle,draw=blue,fill=blue!40!]{v};
        \node (MR) at (2, 0)[circle,draw=blue,fill=blue!40!]{w};
        \node  (BL) at (-4, -2)[circle,draw=blue,fill=blue!20!,minimum size=.8cm]{};
        \node (BM) at (0, -2)[circle,draw=blue,fill=blue!30!]{u};
        \node  (BR) at (4, -2)[circle,draw=blue,fill=blue!20!,minimum size=.8cm]{};
        \begin{scope}[every node/.style={fill=white,circle,minimum width=0pt,inner sep=2pt, outer sep =0pt}]
        \draw (TL) --  (MM);
        \draw (TM) --  (MM);
        \draw (TM) --  (MR);
        \draw (TR) -- (MR);
        \draw (ML) --  (MM);
        \draw (MM) edge node {1}  (MR);
        \draw (BL) edge node {-1}   (MM);
        \draw (BM) -- (MM);
        \draw (BM) edge node {-1}   (MR);
        \draw (BR) -- (MR);
        \end{scope}
    \end{tikzcd} 
     \vspace{.2cm} 
     
    \hspace{2.1cm} \Bigg\downarrow q   

      \vspace{.2cm} 
  \hspace{-.6cm}$(K'', \Lambda''):$\hspace{1cm}  \begin{tikzcd}
        \node  (TL) at (-2, 2)[circle,draw=blue,fill=blue!20!,minimum size=.8cm]{};
        \node  (TM) at (0, 2)[circle,draw=blue,fill=blue!20!,minimum size=.8cm]{};
        \node  (TR) at (2, 2)[circle,draw=blue,fill=blue!20!,minimum size=.8cm]{};
        \node  (ML) at (-2, 0)[circle,draw=blue,fill=blue!20!,minimum size=.8cm]{};
        \node  (MM) at (0, 0)[circle,draw=blue,fill=blue!40!]{v''};
        \node  (BL) at (-2, -2)[circle,draw=blue,fill=blue!20!,minimum size=.8cm]{};
        \node  (BM) at (0, -2)[circle,draw=blue,fill=blue!30!]{u''};
        \node  (BR) at (2, -2)[circle,draw=blue,fill=blue!20!,minimum size=.8cm]{};
        \begin{scope}[every node/.style={fill=white,circle,minimum width=0pt,inner sep=2pt, outer sep =0pt}]
        \draw (TL) -- (MM);
        \draw (TM) -- (MM);
        \draw (TR) -- (MM);
        \draw (ML) -- (MM);
        \draw (BL) edge node {-1} (MM);
        \draw (BM) edge node {-1}  (MM);
        \draw (BR) -- (MM);
        \end{scope}
    \end{tikzcd}
       \end{center}
       \end{example}

The next lemma tells us what happens when $\Lambda$ has no $e$-reduction (at least, in the setting of the Edge-Contraction Lemma). 

\begin{lemma}[Edge-Contraction 2]\label{lem:EdgeContraction2}
   Let $\Lambda$ be an edge marking for $K$. Assume $e \in \E$ is a $1$-marked edge, but all other edges adjacent to $v(e)$ and $w(e)$ are $-1$-marked edges. If $\Lambda$ has no $e$-reduction, then $o^{-1}_K(\Lambda)$ is empty.
\end{lemma}

\begin{proof}
    When no reduction exists there is a single vertex $u \in \V$ so that $\left\{u, v(e) \right\} \in \E$ and $\left\{u, w(e) \right\} \in \E$ are edges, but $\Lambda_{\left\{u, v(e) \right\}} \neq \Lambda_{\left\{u, w(e) \right\}}$. As we just saw at the end of the proof of the Edge Contraction Lemma, the existence of these edges implies
    $$\widetilde{x}_{v(e)} = \pm \widetilde{x}_{w(e)}$$
    for any $x \in \R(K, \SO(3))$. For sake of contradiction, suppose there is some $x \in o^{-1}_K(\Lambda) \subseteq \R(K, \SO(3))$. Then
    $$[\widetilde{x}_u, \widetilde{x}_{v(e)}] = o_K(x)_{\left\{u, v(e) \right\}} = \Lambda_{\left\{u, v(e) \right\}} \neq \Lambda_{\left\{u, w(e) \right\}}  = o_K(x)_{\left\{u, w(e) \right\}}= [\widetilde{x}_u, \widetilde{x}_{w(e)}].$$
    However, since $\widetilde{x}_{w(e)} = \pm \widetilde{x}_{v(e)}$, we have 
    $$[\widetilde{x}_u, \widetilde{x}_{v(e)}]  = [\widetilde{x}_u, \widetilde{x}_{w(e)}],$$
    which is a contradiction. 
\end{proof}

We just saw conditions that guarantee when a fiber of the obstruction map is empty. The next lemma gives conditions under which this fiber is connected. 

\begin{proposition}\label{prop:ConnectedFiber}
    Let $(K, \Lambda)$ be a marked graph so that
    \begin{enumerate}
        \item[(i)] all edges are marked with $-1$, and
        \item[(ii)] there is an ordering $\V = \left\{v_1, ..., v_n\right\}$ of the vertices so that $N_{v_i} \cap \{v_1, ..., v_{i-1}\}$ has cardinality at most $ 2$ for all $i$. 
    \end{enumerate}
    Then $o_K^{-1}(\Lambda)$ is connected.
\end{proposition}

\begin{proof}
    We will use induction on $\vert \V \vert$. Assume $(K, \Lambda)$ satisfies (i) and (ii) and write $v_1, \ldots , v_n$ for the ordering of (ii). Let $K' = K - \{v_n\}$. Then $(K', i^* \Lambda)$ satisfies (i) and (ii) and so $o^{-1}_{K'}(i^* \Lambda)$ is connected by the inductive hypothesis. Consider the map
    \[ i^* : o_{K}^{-1}(\Lambda) \longrightarrow o_{K'}^{-1}(i^* \Lambda) .\]
 It suffices to show the fibers of this map are nonempty and connected.
 
 The fiber over $x \in o_{K'}^{-1}(i^* \Lambda)$ is
    $$(i^*)^{-1}(x) \cong  \left\{ g \in \SO(3) \;  \big|  \; [\widetilde{g}, \widetilde{x}_w] = -1 \text{ for all } w \in N_{v_n} \right\}$$
    where $\pi: S^3 \rightarrow \SO(3)$ is the projection. If $N_{v_n}$ is empty, then this fiber is homeomorphic to $\SO(3)$, which is nonempty and connected. We may assume therefore that $N_{v_n}$ is not empty. By assumption this consists of one or two points. 
    
    By Lemma \ref{lem:AntiCom}, for each $w$ the set of $\widetilde{g} \in S^3$ with $[\widetilde{g} , \widetilde{x}_w]  = -1$ is a great circle in $S^2$; denote this circle by $S_w^1$. Then
    $$(i^*)^{-1}(x) = \pi\left( \bigcap_{w \in N_{v_n}} S_w^1 \right).$$
    If $N_{v_n} = \left\{w_1 \right\}$ consists of one element, then this intersection is the circle $S_{w_1}^1$ and so its projection to $\SO(3)$ is connected. If $N_{v_n} = \left\{w_1, w_2 \right\}$ consists of two elements, then this intersection consists of two antipodal points in $S^2$. The projection to $\SO(3)$ of antipodal points is a single point, which is again connected. 
\end{proof}

\begin{proof}[Proof of Theorem \ref{thm:Inj}]
    The class of graphs under consideration in this theorem is closed under vertex-deletion and edge-contraction. Consider a marking $\Lambda$. We claim that $o_K^{-1}(\Lambda)$ is connected or empty. By the Vertex-Deletion Lemma \ref{lem:VertexDeletion} and the Edge-Contraction Lemma \ref{lem:EdgeContraction}, we can reduce to the case where $\Lambda$ assigns $-1$ to all edges. Since 
    $$o_{K_1 \amalg K_2}^{-1}((\Lambda_1, \Lambda_2)) = o^{-1}_{K_1}(\Lambda_1) \times o^{-1}_{K_2}(\Lambda_2),$$ 
    it suffices to suppose $K$ is connected. For cycles and trees, Proposition \ref{prop:ConnectedFiber} implies that $o^{-1}_K(\Lambda)$ is connected. If $K$ is complete and not a cycle or a path, then $|\V | \geq 4$. In this case $o^{-1}(\Lambda)$ is empty: Any element $x \in o^{-1}(\Lambda)$ would produce 4 elements of $S^3$ that mutually anticommute, which is impossible. 
\end{proof}

\begin{proof}[Proof of Theorem \ref{thm:Surj}]
   It suffices to show that $o^{-1}_K(\Lambda)$ is nonempty for each edge marking $\Lambda$. As in the proof of Theorem \ref{thm:Inj}, we can assume that $K$ is connected and $\Lambda$ assigns $-1$ to all edges. The result follows from Proposition \ref{prop:ConnectedFiber}.

    Conversely, suppose $K$ is not a cycle, but contains a 3-cycle. Then $K$ contains a subgraph $K'$ of the form appearing in Example \ref{ex:1}. As we saw in that example, $K'$ admits an edge marking $\Lambda'$ with $o^{-1}_{K'}(\Lambda')$ empty. Write $i: K' \rightarrow K$ for the inclusion and pick any edge marking $\Lambda$ for $K$ with $i^* \Lambda = \Lambda'$ (since $i$ is injective, $i^*: \pi_1(G)^{\E} \rightarrow \pi_1(G)^{\E'}$ is automatically surjective). Then by Corollary \ref{cor:Funct}, we have a map
    $$i^*: o_K^{-1}(\Lambda) \longrightarrow o_{K'}^{-1}(\Lambda') = \emptyset$$
    and so $o_K^{-1}(\Lambda) $ is empty. 
\end{proof}

\subsection{Proofs of Theorems \ref{thm:2} and \ref{thm:3}, and Corollary \ref{cor:4}}\label{sec:Proofs}

Theorem \ref{thm:2} is an immediate consequence of Theorems \ref{thm:Inj} and \ref{thm:Surj}. We see also from this that the obstruction map $o_K$ detects the components of $\R(K, \SO(3))$. Now suppose $P \rightarrow X$ is as in Theorem \ref{thm:3}. When $\A_{\mathrm{flat}}(P)$ is not empty, we conclude from Corollary \ref{cor:Components} that $\A_{\mathrm{flat}}(P) / \G_0(P)$ is connected. Since $G = \SO(3)$ is connected, it follows from the sequence (\ref{eq:Gauge}) that there is a homeomorphism
$$\A_{\mathrm{flat}}(P) / \G(P) \cong \Big( \A_{\mathrm{flat}}(P) / \G_0(P) \Big) / G$$
and so $\A_{\mathrm{flat}}(P) / \G(P) $ is connected. This proves Theorem \ref{thm:3}.

To prove Corollary \ref{cor:4}, assume that $X$ is a 3-manifold with $\pi_1(X) \cong \Gamma_K$ a RAAG. Droms theorem \cite{Droms} implies that $K$ is a tree or a 3-cycle. Theorem \ref{thm:Surj} tells us the obstruction map is surjective with $G = \SO(3)$, so Corollary \ref{cor:Surj} implies that all principal $\SO(3)$-bundles $P$ on $X$ admit flat connections. Thus $\A_{\mathrm{flat}}(P)/ \G(P)$ is not empty, and its connectedness is a consequence of Theorem \ref{thm:3}. \qed


\bibliographystyle{alpha}

\begin{thebibliography}{10}

\bibitem{AM}
S.\ Akbulut, J.\ D.\ McCarthy, 
\emph{Casson's Invariant for Oriented Homology Three-Spheres: An Exposition}.
Mathematical Notes, 36. Princeton University Press, Princeton, NJ, 1990.

\bibitem{Agol2}
I.\ Agol,
Criteria for virtual fibering,
\emph{J. of Top.}, Vol. 1, Iss. 2, (2008) 269--284.

\bibitem{Agol}
I.\ Agol,
{The virtual Haken conjecture (with an appendix by Ian Agol, Daniel Groves and Jason Manning).}
\emph{Documenta Mathematica} 18 (2013) 1045--1087.

\bibitem{Brau}
A. Baudisch, 
Subgroups of semifree groups, 
\emph{Acta Math. Acad. Sci. Hungar.} 38 (1981), no. 1--4, 19--28.




\bibitem{CW}
J.\ Crisp and B.\ Wiest,
{Embeddings of graph braid and surface groups in right-angled Artin groups and braid groups.}
\emph{Algebr. Geom. Topol.} 4 (2004) 439–472.

\bibitem{Dehn}
M.\ Dehn, \"{U}ber die Topologie des dreidimensional Raumes,
\emph{Math. Ann.} 69 (1910), 137--168.

\bibitem{DS}
S.\ Dostoglou, D.\ Salamon,
Instanton homology and symplectic fixed points,
\emph{Symplectic geometry}, 57--93, \emph{London Math. Soc. Lecture Note} Ser., 192, Cambridge Univ. Press, Cambridge, 1993.

\bibitem{Droms}
C. Droms,
{Graph groups, coherence, and three-manifolds},
\emph{J. Alg.}, Vol. 106, Iss. 2 (1987) 484--489.

\bibitem{FL}
C.\ Florentino, S.\ Lawton,
Topology of character varieties of Abelian groups,
\emph{Top. and its App.} Vol. 173, Iss. 15 (2014) 32--58.

\bibitem{FL2}
C.\ Florentino, S.\ Lawton,    
{Flawed groups and the topology of character varieties},
\emph{Topology and its Applications} Vol. 341 (2024).

\bibitem{Kob}
T.\ Koberda,
{Right-angled Artin groups and a generalized isomorphism problem for finitely generated subgroups of mapping class groups.}
\emph{Geom. Funct. Anal.} 22, no. 6 (2012) 1541--1590.

\bibitem{KN}
S.\ Kobayashi, K.\ Nomizu,
Foundations of Differential Geometry. Wiley Classics Edition, Published 1996. 
John Wiley \& Sons Inc. New York, Chichester, Brisbane, Toronto, Singapore. 1963.


\bibitem{Wang}
S.\ Wang,
{Representations of surface groups and right-angled Artin groups in higher rank}.
\emph{Alg. \& Geo. Top.} 7 (2007) 1099--1117.


\bibitem{Zent}
R.\ Zentner,
Integer homology 3-spheres admit irreducible representations in $\SL(2, \bb{C})$. 
\emph{Duke Math. J.} 167 (9) (2018) 1643--1712.






\end{thebibliography}

\end{document}